\documentclass[letterpaper, 10 pt, conference]{ieeeconf}
\usepackage[T1]{fontenc}
\IEEEoverridecommandlockouts                          

\usepackage{amsmath} 
\usepackage{amssymb}  
\usepackage{xcolor}
\usepackage{caption}
\usepackage{graphicx}
\usepackage{mathtools}
\usepackage{caption}
\usepackage{floatrow}
\usepackage{tikz}
\usepackage{tikz-cd}
\usepackage{url}

\usepackage{algorithmicx}
    \usepackage[ruled]{algorithm}
    \usepackage{algpseudocode}

\usepackage{pgfplots}
\usepackage{xcolor}
\usetikzlibrary{matrix,arrows,calc,positioning,shapes,decorations.pathreplacing}
\usepackage{graphicx}

\usepackage{amsmath} 

\usepackage{enumerate}
\usepackage[all,tips]{xy}
\SelectTips{cm}{11}

\newtheorem{corollary}{Corollary}
\newtheorem{definition}{Definition}
\newtheorem{remark}{Remark}
\newtheorem{proposition}{Proposition}

\newtheorem{example}{Example}

\newcommand{\R}{\mathbb{R}}

\newcommand{\ra}{\rightarrow}

\newcommand{\lcf}{\lbrack\! \lbrack}
\newcommand{\rcf}{\rbrack\! \rbrack}

\usepackage{amssymb}

\newcommand{\proa}{A^*G \mbox{$\;$}_{\tau^*} \kern-3pt\times_\alpha
G \mbox{$\;$}_\beta \kern-3pt\times_{\tau^*} A^*G}

\def\lcf{\lbrack\! \lbrack}
\def\rcf{\rbrack\! \rbrack}

\begin{document}

\title{
Retraction maps in optimal control of nonholonomic systems}

\author{Alexandre Anahory Simoes, Mar\'ia Barbero Liñ\'an, Anthony Bloch, Leonardo Colombo, David Mart\'in de Diego. 
\thanks{A. Anahory Simoes (alexandre.anahory@ie.edu) is with the School of Science and Technology, IE University, Spain.}
\thanks{ M. Barbero  Li\~n\'an (m.barbero@upm.es) is with Departamento de Matem\'atica Aplicada, Universidad Polit\'ecnica de Madrid, Av. Juan de Herrera 4, 28040 Madrid, Spain.}
\thanks{A. Bloch (abloch@umich.edu) is with Department of Mathematics, University of Michigan, Ann Arbor, MI 48109, USA.}
\thanks{L. Colombo (leonardo.colombo@car.upm-csic.es) is with Centre for Automation and Robotics (CSIC-UPM), Ctra. M300 Campo Real, Km 0,200, Arganda
del Rey - 28500 Madrid, Spain.} \thanks{David Mart\'in de Diego (david.martin@icmat.es) is with the Institute of Mathematical Sciences (CSIC-UAM-UCM-UC3M). Calle Nicol\'as Cabrera 13-15, Cantoblanco, 28049, Madrid, Spain.} \thanks{The authors acknowledge financial support from Grants  PID2022-137909NB-C21  and  CEX2023-001347-S  funded by MCIN/AEI/10.13039/501100011033. A.B. was partially supported by NSF grant  DMS-2103026, and AFOSR grants FA
9550-22-1-0215 and FA 9550-23-1-0400.
}}%

\maketitle

\begin{abstract}
    In this paper, we compare the performance of different numerical schemes in approximating Pontryagin's Maximum Principle's necessary conditions for the optimal control of nonholonomic systems. Retraction maps are used as a seed to construct geometric integrators for the corresponding Hamilton equations. First, we obtain an intrinsic formulation of a discretization map on a distribution $\mathcal{D}$. Then, we illustrate this construction on a particular example for which the performance of different symplectic integrators is examined and compared with that of non-symplectic integrators.
\end{abstract}


\section{Introduction}

Nonholonomic control systems are characterized by non-integrable constraints on their velocities. The optimal control problem for a nonholonomic system is typically formulated as minimizing a cost functional subject to nonholonomic dynamics. Pontryagin’s Maximum Principle (PMP) provides first-order necessary conditions for optimality by introducing costate variables for the dynamics and additional constraints (see \cite{B} and references there in). PMP yields a Hamiltonian system on the cotangent bundle of the state space that the optimal trajectory must satisfy. 

Despite the Hamiltonian character of PMP equations, the question of whether symplectic integrators perform better than non-symplectic integrators is difficult to answer. It usually depends on additional structure of each problem. The report \cite{chyba} explores the same problem in the particular case of the Martinet system, an example of a kinematic nonholonomic control system.


There are two main families of methods to approach the numerical approximation of this problem: either one discretizes both the cost function and the control dynamics and tries to solve the resulting discrete nonlinear minimization problem; or one obtains necessary conditions for the optimal trajectory, then discretizes the optimality conditions. The former are called direct methods, the latter are called indirect methods. We will focus on the latter. 



The notion of retraction map is an essential tool in different research areas like optimization theory, numerical analysis and interpolation (see \cite{AbMaSeBookRetraction} and references therein). A retraction map allows to generalize linear-search methods in Euclidean spaces to general manifolds by providing an intrinsic framework to move between two points on a manifold in the direction of a prescribed tangent vector. In the Euclidean setting, this is the principle behind approximating the time derivative of a trajectory using finite differences. That is why retraction maps have been widely used to construct numerical integrators of ordinary differential equations on arbitrary manifolds.

In \cite{21MBLDMdD}, the classical notion of retraction map is extended to the notion of discretization maps which are used to construct geometric integrators. Indeed, using the geometry of the tangent and cotangent bundles, the authors show not only how to construct well-known symplectic integrators from a choice of a discretization map, but also how to obtain new ones by choosing more sophisticated maps. Such maps were further applied in \cite{DM} to construct numerical methods for optimal control problems from a Hamiltonian perspective. 


The goal of this paper is to use the notion of discretization map given in~\cite{21MBLDMdD} to construct symplectic integrators for the PMP equations satisfied by a solution of an optimal control problem of a nonholonomic mechanical systems, and to study the performance of symplectic versus non-symplectic integrators for the optimal control of nonholonomic control dynamical systems (see \cite{B, BlCoGuMdD, BlochCrouch2}). In particular, we compare their performance in preserving the Hamiltonian function and the nonholonomic constraints in a specific example. In future research, we will show how to extend these methods to optimal control problems of nonholonomic systems on Lie groups and homogeneous spaces.

The paper is structured as follows: in Section II we recall some preliminaries from differential geometry. In Section III, we recall the basic concepts about nonholonomic systems and we recall the concept of adapted coordinates to a distribution that will be useful to compute retraction maps tailored to nonholonomic control systems. In Section IV, we state the optimal control of nonholonomic mechanical systems and rewrite it in terms of adapted coordinates. In Section \ref{sec:retractions:D}, we define a retraction map on the distribution $\mathcal{D}$ determining the nonholonomic constraints. Finally, in Section \ref{sec:example}, we examine the example of the nonholonomic robot in detail, compute explicit expressions for the retraction map and the associated integrator, whose performance is subsequently discussed.


\section{Preliminaries}\label{sec2}

Let $Q$ be a $n$-dimensional differentiable configuration
manifold of a mechanical system with local coordinates $(q^i)$, $1\leq i\leq n$. Denote by $TQ$ the
tangent bundle (see, for instance, \cite{marsden} for an introduction to the tangent bundle and mechanics on it). If $T_{q}Q$ denotes the tangent space of $Q$ at the point $q$, then $\displaystyle{TQ:=\cup_{q\in Q}T_{q}Q}$, with induced local coordinates $(q^i, \dot{q}^i)$.  There is a canonical projection $\tau_{Q}:TQ \rightarrow Q$, sending each vector $v_{q}$ to the corresponding base point $q$ as follows $\tau_{Q}(q^{i},\dot{q}^{i})=(q^{i})$. 

The vector space structure of $T_{q}Q$ makes possible to consider its dual space, $T^{*}_{q}Q$, to define the cotangent bundle as $\displaystyle{T^{*}Q:=\cup_{q\in Q}T^{*}_{q}Q},$ with local coordinates $(q^i,p_i)$.  There is a canonical projection $\pi_{Q}:T^{*}Q \rightarrow Q$, sending each momenta $p_{q}$ to the corresponding base point $q$. Note that in coordinates $\pi_{Q}(q^{i},p_{i})=(q^{i})$.



\subsection{Geometry of Hamiltonian systems}

A Hamiltonian function $H:T^{*}Q\to\mathbb{R}$ is described by the total energy of a mechanical system and leads to Hamilton's equations on $T^{*}Q$, whose solutions are integral curves of the Hamiltonian vector field $X_H$ taking values in $T(T^{*}Q)$ associated with $H$. 
Locally, $X_H(q,p)=\left(\frac{\partial H}{\partial p},-\frac{\partial H}{\partial q}\right)$, that is,
\begin{equation}\label{hameq1}\dot{q}^{i}=\frac{\partial H}{\partial p_i},\quad\dot{p}_{i}=-\frac{\partial H}{\partial q^i},\quad 1\leq i\leq n.\end{equation} Equations~\eqref{hameq1} determine a set of $2n$ first order ordinary differential equations (see \cite{B}, for instance, for more details).

A one-form $\alpha$ on $Q$ is a map assigning to each point $q$ a cotangent vector on $q$, that is, $\alpha(q)\in T^{*}_qQ$. Cotangent vectors acts linearly on vector fields according to $\alpha(X) = \alpha_{i}X^{i}\in \mathbb{R}$ if $\alpha = \alpha_{i}dq^{i}$ and $X = X^{i} \frac{\partial}{\partial q^{i}}$. Analogously, a two-form or a $(0,2)$-tensor field is a bilinear map that acts on a pair of vectors  to produce a number. 

A symplectic form $\omega$ on a manifold $Q$ is a $(0,2)$-type tensor field that is skew-symmetric and non-degenerate, i.e., $\omega(X,Y)=-\omega(Y,X)$ for all vector fields $X$ and $Y$ and if $\omega(X,Y)=0$ for all vector fields $X$, then $Y\equiv 0$. 

The set of vector fields and the set of 1-forms on $Q$ are denoted by $\mathfrak{X}(Q)$ and $\Omega^{1}(Q)$, respectively. The symplectic form induces a linear isomorphism $\flat_{\omega}:\mathfrak{X}(Q)\rightarrow \Omega^{1}(Q)$, given by $\langle\flat_{\omega}(X),Y\rangle=\omega(X,Y)$ for any vector fields $X, Y$. The inverse of $\flat_{\omega}$ will be denoted by $\sharp_{\omega}$.

As described in~\cite{LiMarle}, the cotangent bundle $T^*Q$ of a  differentiable manifold $Q$ is equipped with a canonical exact symplectic structure $\omega_Q=-d\theta_Q$, where $\theta_Q$ is the canonical 1-form on $T^*Q$. In canonical bundle coordinates $(q^i, p_i)$ on $T^*Q$, 
$
\theta_Q= p_i\, \mathrm{d}q^i$ and
$\omega_Q= \mathrm{d}q^i\wedge \mathrm{d}p_i\; .
$
Hamilton's equations can be intrinsically rewritten as $
\imath_{X_H}\omega_Q\colon = \flat_\omega(X_H)=\mathrm{d}H\; .
$
Hamiltonian dynamics are characterized by the following two essential properties~\cite{hairer}:
\begin{itemize}
	\item  Preservation of energy by the Hamiltonian function:  $$0=\omega_Q(X_H, X_H)=dH(X_H)=X_H(H)\,.$$
	
	\item Preservation of the symplectic form: 
 If $\{\phi^t_{X_H}\}$ is the flow of $X_H$, then the pull-back of the differential form by the flow is preserved, $(\phi^t_{X_H})^*\omega_Q=\omega_Q$.

\end{itemize}

 Recall that a
pair $(Q,\omega)$ is called a symplectic manifold if $Q$ is a
differentiable manifold and  $\omega$ is a symplectic 2-form. As a consequence, the restrictions of $\omega$ to each $q\in Q$ makes the tangent space $T_{q}Q$ into a symplectic vector space.

\begin{definition}
Let $(Q_1,\omega_1)$ and
$(Q_2,\omega_{2})$ be two symplectic manifolds, let $\phi:Q_1\to Q_2$ be a smooth map. The
map $\phi$ is called symplectic if the symplectic forms are preserved:
$
\phi^*\omega_2=\omega_1$. Moreover, it is a \textit{symplectomorphism} if $\phi$ is a diffeomorphism and $\phi^{-1}$ is also
symplectic.
\end{definition}

		Let $Q_1$ and $Q_2$ be $n$-dimensional manifolds and $F: Q_1\rightarrow Q_2$ be a smooth map. The
	{\it tangent lift} $TF: TQ_1\rightarrow TQ_2$  of $F$ is defined by $TF(v_q)=T_qF (v_q)\in T_{F(q)} Q_2\, \, \mbox{ where } v_q\in T_qQ_1$, and $T_qF$ is the tangent map of $F$ whose matrix is the Jacobian matrix of $F$ at $q\in Q_1$.

		As the tangent map $T_qF$ is linear, the dual map $T_{q}^*F\colon T^*_{F(q)}Q_2\rightarrow T^*_qQ_1$ is defined as follows:
  \[\langle(T^*_{q}F)(\alpha_2), v_{q}\rangle=\langle \alpha_2, T_{q}F(v_{q})\rangle\mbox{ for every } v_q\in T_qQ_1.\]
  Note that $(T^*_{q}F)(\alpha_2)\in T^*_qQ_1$. 
  
 \begin{definition}\label{def:colift}
Let $F: Q_1\rightarrow Q_2$ be a diffeomorphism. The vector bundle morphism 
	$\widehat{F}: T^*Q_1\rightarrow T^*Q_2$ defined by	$\widehat{F}=T^*F^{-1}$
	is called the cotangent lift of $F^{-1}$. 
	 \end{definition}
In other words, $\widehat{F}(\alpha_q)= 	T^*_{F(q)}F^{-1} (\alpha_q)$ where $\alpha_q\in T^*_q Q_1$. Obviously, $(T^*F^{-1})\circ (T^*F)={\rm Id}_{T^*Q_2}$.

\subsection{Discretization maps}\label{sec4}

The first notion of retraction map in the literature can be found in~\cite{1931Borsuk} from a topological viewpoint. Later on, the notion of retraction map as defined below is used to obtain Newton's method on Riemannian manifolds~\cite{1986Shub,2002Adler}.

\begin{definition}\label{def-RetractMap} A \textit{retraction map} on a manifold $Q$ is a smooth mapping $R$ from the tangent bundle $TQ$ onto $Q$. Let $R_q$ denote the restriction of $R$ to $T_qQ$, the following properties are satisfied:
\begin{enumerate}
	\item $R_q(0_q)=q$, where $0_q$ denotes the zero element of the vector space $T_qQ$.
	\item With the canonical identification $T_{0_q}T_qQ\simeq T_qQ$, $R_q$ satisfies \begin{equation}\label{eq-DefRetract-prop}
	{\rm D}R_q(0_q)=T_{0_q}R_q={\rm Id}_{T_qQ},
	\end{equation}
	where ${\rm Id}_{T_qQ}$ denotes the identity mapping on $T_qQ$.
\end{enumerate}
\end{definition}

The condition~\eqref{eq-DefRetract-prop} is known as \textit{local rigidity condition} since, given $\xi\in T_qQ$,  the curve $\gamma_\xi(t)=R_q(t\xi)$ has $\xi$ as tangent vector at $q$, i.e. $\dot{\gamma}_\xi(t)= \langle DR_q(t\xi), \xi\rangle\; \hbox{ and, in consequence}, \dot{\gamma}_\xi(0)= {\rm Id}_{T_qQ}(\xi)=\xi$.

A typical example of a retraction map is the exponential map, $\hbox{exp}$, on Riemannian manifolds given in~\cite[Chapter 3.2]{doCarmo}. Therefore, the image of $\xi$ through the exponential map is a point on the Riemannian manifold $(Q,\mathcal{G})$, where $\mathcal{G}$ is a Riemannian metric, obtained by moving along a geodesic a length equal to the norm of $\xi$ starting with the velocity $\xi/\|\xi\|$, that is, 
\[
\hbox{exp}_q(\xi)=\sigma (\|\xi\|)\; ,
\]
where $\sigma$ is the unit speed geodesic such that $\sigma(0)=q$ and $\dot{\sigma}(0)=\xi/\|\xi\|$. 

Next, we define a generalization of the retraction map in Definition~\ref{def-RetractMap} that allows a discretization of the tangent bundle of the configuration manifold leading to the construction of numerical integrators as described in~\cite{21MBLDMdD}. Given a point and a velocity, we obtain two nearby points that are not necessarily equal to the initial base point. 

\begin{definition} \label{def:DiscreteMap2} A map 
	$R_d\colon U\subset TQ\rightarrow Q\times Q$ given by \begin{equation*}
	R_d(q,v)=(R^1(q,v),R^2(q,v)),
	\end{equation*} 
	where  $U$ is an open neighborhood of the zero section $0_q$ of $TQ$, 
	defines a {\it discretization map on $Q$} if it satisfies 
	\begin{enumerate}
		\item $R_d(q,0)=(q,q)$,
		\item $T_{0_q}R^2_q-T_{0_q}R^1_q\colon T_{0_q}T_qQ\simeq T_qQ\rightarrow T_qQ$ is equal to the identity map on $T_qQ$ for any $q$ in $Q$, where $R^a_q$ denotes the restrictions of $R^a$ to $T_qQ$ for $a=1,2$.
	\end{enumerate}
\end{definition}
Thus, the discretization map $R_d$ is a local diffeomorphism from some neighborhood of the zero section of $TQ$.

If $R^1(q,v)=q$, the two properties in Definition~\ref{def:DiscreteMap2} guarantee that both properties in Definition~\ref{def-RetractMap} are satisfied by $R^2$. Thus, Definition~\ref{def:DiscreteMap2} generalizes Definition~\ref{def-RetractMap}. 

		

\begin{example} The midpoint rule on an Euclidean vector space can be recovered from the following discretization map:
$R_d(q,v)=\left( q-\tfrac{v}{2}, q+\tfrac{v}{2}\right).$

\end{example}


\subsection{Cotangent lift of discretization maps}

As the Hamiltonian vector field takes values on $TT^*Q$, the discretization map must be on $T^*Q$, that is, 
$R^{T^*}_d: TT^*Q \rightarrow T^*Q\times T^*Q$. Such a map is obtained by cotangently lifting a discretization map $R_d\colon TQ\rightarrow Q\times Q$, so that the construction $R^{T^*}_d$ is a symplectomorphism. In order to do that, we need the following three symplectomorphisms (see \cite{21MBLDMdD} and \cite{DM} for more details): 
\begin{itemize}
\item The cotangent lift of the diffeomorphism $R_d\colon TQ\rightarrow Q\times Q$ as described in Definition~\ref{def:colift}. 
	\item The canonical symplectomorphism:
	\begin{equation*} \alpha_Q\colon TT^*Q  \longrightarrow  T^*TQ  
	\end{equation*}  \mbox{ such that } $\alpha_Q(q,p,\dot{q},\dot{p})= (q,\dot{q},\dot{p}, p).$

	\item  The symplectomorphism between $(T^*(Q\times Q), \omega_{Q\times Q})$     and   
	$(T^*Q\times T^*Q, \Omega_{12}:=pr_2^*\omega_Q-pr^*_1\omega_Q)$:
		\begin{equation*}
	\Phi:T^*Q\times T^*Q \longrightarrow T^*(Q\times Q)\; , \; 
		\end{equation*} given by $\Phi(q_0, p_0; q_1, p_1)=(q_0, q_1, -p_0, p_1).$ Here, $pr_{i}:T^{*}Q\times T^{*}Q \to T^{*}Q$ is the projection to the $i$-th factor.
	\end{itemize}
Diagram in Fig.~\ref{Diag:RdT*} summarizes the construction procress from $R_d$ to $R_d^{T^*}$:
	\begin{figure} [htb!]
 \begin{equation*}
\xymatrix{ {{TT^*Q }} \ar[rr]^{{{R_d^{T^*}}}}\ar[d]_{\alpha_{Q}} && {{T^*Q\times T^*Q }}  \\ T^*TQ \ar[d]^{\pi_{TQ}}\ar[rr]^{	\widehat{R_d}}&& T^*(Q\times Q)\ar[u]_{\Phi^{-1}}\ar[d]^{\pi_{Q\times Q}}\\ TQ \ar[rr]^{R_d} && Q\times Q }
\end{equation*}
    \caption{Definition of the cotangent lift of a discretization.}\label{Diag:RdT*}
\end{figure}

\begin{proposition}\cite{21MBLDMdD}
	Let $R_d\colon TQ\rightarrow Q\times Q$ be a discretization map on $Q$. Then $${{R_d^{T^*}=\Phi^{-1}\circ \widehat{R_d}\circ \alpha_Q\colon TT^*Q\rightarrow T^*Q\times T^*Q}}$$ 
is a discretization map  on $T^*Q$.\label{Prop:RdT*}
\end{proposition}
\begin{corollary}\cite{21MBLDMdD} The discretization map
 ${{R_d^{T^*}}}=\Phi^{-1}\circ 	(TR_d^{-1})^* \circ \alpha_Q\colon T(T^*Q)\rightarrow T^*Q\times T^*Q$ is a symplectomorphism between $(T(T^*Q), {\rm d}_T \omega_Q)$ and $(T^*Q\times T^*Q, \Omega_{12})$.
\end{corollary}
\begin{example}\label{example3} On $Q={\mathbb R}^n$ the discretization map 
	$R_d(q,v)=\left(q-\frac{1}{2}v, q+\frac{1}{2}v\right)$ is cotangently lifted to
		$$R_d^{T^*}(q,p,\dot{q},\dot{p})=\left( q-\dfrac{1}{2}\,\dot{q}, p-\dfrac{\dot{p}}{2}; \; q+\dfrac{1}{2}\, \dot{q}, p+\dfrac{\dot{p}}{2}\right)\, .$$
\end{example}

\section{Description of Nonholonomic Dynamics}

Throughout the paper, we will consider the case of nonholonomic mechanical systems where the Lagrangian is of mechanical type, that is, the associated Lagrangian function $L:TQ\to\R$ is of the form \[
L(v_q)=\frac{1}{2}\mathcal{G}(v_q, v_q) - V(q),
\]
with $v_q\in T_qQ$, where $\mathcal{G}$ denotes a Riemannian metric on the configuration
space $Q$, $V:Q\ra\R$ is a potential function, and the velocities are restricted to belong to a non-integrable regular distribution $\mathcal{D}$ on $Q$.
Locally, the metric is determined by the matrix $M=(\mathcal{G}_{ij})_{1\leq i, j\leq n}$
where $\mathcal{G}_{ij}=\mathcal{G}(\partial/\partial q^i,\partial/\partial q^j)$.

Next, assume that the system is subject to nonholonomic constraints, defined by a regular distribution $\mathcal{D}$ on $Q$ with corank$(\mathcal{D})=m$.
Denote by $\tau_{\mathcal{D}}:\mathcal{D}\ra Q$ the canonical
projection of $\mathcal{D}$ onto $Q$ and by 
$\Gamma(\tau_{\mathcal{D}})$ the set of sections of $\tau_{D}$, which
in this case is just the set of vector fields $\mathfrak{X}(Q)$
taking values on $\mathcal{D}.$ If $X, Y\in\mathfrak{X}(Q),$ then
$[X,Y]$ denotes the standard Lie bracket of vector fields.

For $X,Y\in\Gamma(\tau_{\mathcal{D}})$, it may happen that $[X,Y]\notin\Gamma(\tau_{\mathcal{D}})$
because $\mathcal{D}$ is non-integrable. If so, the space of sections of $\mathcal{D}$ is not closed under the usual Lie bracket. However, we can modify the Lie bracket of vector fields to obtain a bracket on sections of $\mathcal{D}$. Using the Riemannian metric $\mathcal{G}$, we can define two complementary orthogonal projectors ${\mathcal P}\colon TQ\to {\mathcal D}$ and ${\mathcal Q}\colon TQ\to {\mathcal
D}^{\perp},$ with respect to the orthogonal decomposition $\mathcal{D}\oplus\mathcal{D}^{\perp}=TQ$. Therefore, given $X,Y\in\Gamma(\tau_{\mathcal{D}})$, the
\textit{nonholonomic bracket} is defined by
$\lcf\cdot,\cdot\rcf:\Gamma(\tau_{\mathcal{D}})\times\Gamma(\tau_{\mathcal{D}})\rightarrow\Gamma(\tau_{\mathcal{D}})$
as
$$\lcf X,Y\rcf:=\mathcal{P}[X,Y].$$ This Lie bracket verifies the usual properties of a Lie bracket except the Jacobi identity (see \cite{Grab,BlCoGuMdD} and references therein). 

A curve $\gamma:I\subset\R\ra\mathcal{D}$ is \textit{admissible} if
$$\frac{d\sigma}{dt}(t)=\gamma(t),$$ where $\tau_{\mathcal{D}}\circ\gamma=\sigma$.

Given local coordinates on $Q,$ $(q^{i})$ with $i=1,\ldots,n;$ and
a local basis of sections of $\mathcal{D}$ denoted by $\{e_{A}\}$, with $A=1,\ldots,n-m$, such that
$\displaystyle{e_{A}=\rho_{A}^{i}(q)\frac{\partial}{\partial
q^{i}}}$ we introduce \textit{adapted coordinates} $(q^{i},y^{A})$ on
$\mathcal{D}$, where, if $e\in\mathcal{D}_{x}$ then
$e=y^{A}e_{A}(x).$ Therefore, $\gamma(t)=(q^{i}(t),y^{A}(t))$ is
admissible if
$$\dot{q}^{i}(t)=\rho_{A}^{i}(q(t))y^{A}(t).$$ In addition, we can express the modified Lie bracket as a linear combination of the local sections $e_{C}$, i.e., $\lcf e_A,e_B\rcf=c_{AB}^{C}e_{C}$ where the coefficients $c_{AB}^{C}$ are called the structure functions of the Lie bracket.

Consider the restricted Lagrangian function
$\ell:\mathcal{D}\rightarrow\mathbb{R},$
\begin{equation}\label{eq:lrestrict} \ell(v)=\frac{1}{2}\mathcal{G}^{\mathcal{D}}(v,v)-V(\tau_{D}(v)),\hbox{ with }  v\in\mathcal{D},\end{equation}
where $\mathcal{G}^{\mathcal{D}}:\mathcal{D}\times_{Q}\mathcal{D}\ra\R$ is the restriction of the Riemannian metric $\mathcal{G}$ to
the distribution $\mathcal{D}$.

\begin{definition}[\cite{BlCoGuMdD}]
A \textit{solution of the nonholonomic problem} determined by the Lagrangian function $\ell$ and the distribution $\mathcal{D}$ is an admissible
curve $\gamma:I\rightarrow\mathcal{D}$ with local coordinates $(q^{i}(t), y^{A}(t))$ satisfying the equations
\begin{equation}\label{eq:lrestricted}\begin{array}{rcl}
\dot{q}^{i}&=&\rho_{A}^{i}y^{A},\\ [1mm]
\dfrac{d}{dt}\dfrac{\partial \ell}{\partial y^{A}}&=&c_{BA}^{C}\dfrac{\partial \ell}{\partial y^{C}}y^{B} + \rho^{i}_{A} \dfrac{\partial \ell}{\partial q^{i}}.
\end{array}
\end{equation}

\end{definition}

\begin{remark}
The nonholonomic  equations only depend on the coordinates
$(q^{i},y^{A})$ on $\mathcal{D}$.  Therefore, the nonholonomic
equations are free of Lagrange multipliers. These equations are the  \textit{nonholonomic Hamel equations} (see
\cite{BloZen}, for example, and references therein).
\end{remark}

For the mechanical Lagrangian function $\ell$ above, Hamel equations end up having the following form
\begin{eqnarray*}
\dot{q}^{i}&=&\rho_{A}^{i}(q)y^{A},\\
\dot{y}^{C}&=&-\Gamma_{AB}^{C}y^{A}y^{B}-(\mathcal{G}^{\mathcal{D}})^{CB}\rho_{B}^{i}\frac{\partial V}{\partial q^{i}},
\end{eqnarray*}
where $(\mathcal{G}^{\mathcal{D}})^{AB}$ denotes the coefficients of the inverse matrix of $(\mathcal{G}^{\mathcal{D}})_{AB}$, with $\mathcal{G}^{\mathcal{D}}(e_{A},e_{B})=(\mathcal{G}^{\mathcal{D}})_{AB}$, and $\Gamma_{AB}^{C}$ are functions of the variables $q^{i}$ aggregating the quadratic terms in $y^{A}$. In fact, $\Gamma_{AB}^{C}$ are the Christoffel symbols associated with a Levi-Civita connection induced by the fibered metric $\mathcal{G}^{\mathcal{D}}$ (see \cite{BlCoGuMdD} for more details).

\section{Optimal control of nonholonomic mechanical
systems}\label{section3}

The purpose of this section is to study optimal control problems for nonholonomic mechanical systems. We shall assume that all the control systems under consideration are controllable in the configuration space, that is, for any two
points $q_0$ and $q_f$ in the configuration space $Q$, there exists
an admissible control $u(t)$ defined on the control manifold
$U\subseteq\R^{n}$ such that the system with initial condition $q_0$
reaches the point $q_f$ at time $T$ (see \cite{B} for
more details).
We analyze the case when the dimension of the input or control distribution is equal to the rank of $\mathcal{D}$. If the rank of $\mathcal{D}$ is equal to the dimension of the control distribution, the system will be called a \textit{fully actuated nonholonomic system}.

\begin{definition}
A \textit{solution of a fully actuated nonholonomic problem} is an
admissible curve $\gamma:I\rightarrow\mathcal{D}$ such that it satisfies: 
$$\frac{d}{dt}\frac{\partial \ell}{\partial y^{A}}=c_{BA}^{C}\frac{\partial \ell}{\partial y^{C}}y^{B} + \rho^{i}_{A} \frac{\partial \ell}{\partial q^{i}} + u^{A}
.$$
where $u^{A}$ are the control inputs.
\end{definition}

For the particular Lagrangian function $\ell$ in Equation~\eqref{eq:lrestrict} the above equations become
\begin{eqnarray*}
\dot{q}^{i}&=&\rho_{A}^{i}y^{A},\\
\dot{y}^{C}&=&-\Gamma_{AB}^{C}y^{A}y^{B}-(\mathcal{G}^{\mathcal{D}})^{CB}\rho_{B}^{i}\frac{\partial V}{\partial q^{i}}+u^{C}.
\end{eqnarray*}

For a cost function
\begin{eqnarray*}
C&:&\mathcal{D}\times U\rightarrow\mathbb{R}\\
&&(q^{i},y^{A}, u^A)\mapsto C(q^{i},y^{A},u^{A})
\end{eqnarray*} the \textit{optimal control problem} consists of finding an admissible curve
$\gamma:I\rightarrow\mathcal{D}$ which is a solution of the fully actuated
nonholonomic problem given initial and final boundary conditions on
$\mathcal{D}$ and minimizing the cost functional
$$\mathcal{J}(\gamma(t),u(t)):=\int_{0}^{T}C(\gamma(t),u(t))dt.$$

Define the submanifold $\mathcal{D}^{(2)}$ of $T\mathcal{D}$ by
\begin{equation*}
\mathcal{D}^{(2)}:=\{v\in T\mathcal{D}\mid v=\dot{\gamma}(0), \gamma:I\rightarrow\mathcal{D} \text{ admissible curve} \}.
\end{equation*}
 We can choose coordinates $(q^{i},y^{A},\dot{y}^{A})$ on
$\mathcal{D}^{(2)}$, where the inclusion on $T\mathcal{D}$,
$i_{\mathcal{D}^{(2)}}:\mathcal{D}^{(2)}\hookrightarrow
T\mathcal{D}$, is given by $i_{\mathcal{D}^{(2)}}(q^{i},y^{A},\dot{y}^{A})=(q^{i},y^{A},\rho_{A}^{i}(q)y^{A},\dot{y}^{A})$. Therefore, $\mathcal{D}^{(2)}$ is locally described by the following constraints on $T\mathcal{D}$ $$\dot{q}^{i}-\rho_{A}^{i}(q)y^{A}=0.$$

Our optimal control problem is alternatively
determined by a function
$\mathcal{L}:\mathcal{D}^{(2)}\rightarrow\mathbb{R}$, where

\begin{small}\begin{equation}\label{lagrangiancontrol}\mathcal{L}(q^{i},y^{A},\dot{y}^{A})=C\left(q^{i},
y^{A},\dot{y}^{A}+\Gamma_{CB}^{A}y^{C}y^{B}+(\mathcal{G}^{\mathcal{D}})^{AB}\rho_{B}^{i}\frac{\partial
V}{\partial q^{i}}\right).
\end{equation}\end{small}

\begin{definition}
The second-order nonholonomic system on 
$\mathcal{D}^{(2)}$ is \textit{regular} if the
matrix
$\displaystyle{\left(\frac{\partial^{2}\mathcal{L}}{\partial\dot{y}^{A}\partial\dot{y}^{B}}\right)}$
is non singular.
\end{definition}

\begin{remark}
Observe that if the Lagrangian $\mathcal{L}:\mathcal{D}^{(2)}\ra\R$
is determined from an optimal control problem and its expression is
given by \eqref{lagrangiancontrol}, then the regularity of the matrix
$\displaystyle{\left(\frac{\partial^{2}\mathcal{L}}{\partial\dot{y}^{A}\partial\dot{y}^{B}}\right)}$
 is equivalent to $$\det\left(\frac{\partial^{2}C}{\partial u^{A}\partial u^{B}}\right)\neq
 0$$ for the cost function $C$.

\end{remark}

Assume that the system is regular. If $\displaystyle{p_{A}=\tfrac{\partial\mathcal{L}}{\partial\dot{y}^{A}}}$,
we can write $\dot{y}^{A}=\dot{y}^{A}(q^{i},y^{A},p_{A})$.  Let $(q^{i},y^{A},p_{i}, p_A)$ be the induced coordinates on $T^{*}\mathcal{D}$, the Hamiltonian function is locally defined by
\begin{equation}\label{Hamiltonian:D}
    \begin{split}
        \mathcal{H}(q^{i},y^{A},p_{i},p_{A})= & p_{A}\dot{y}^{A}(q^{i},y^{A},p_{A})+p_{i}\rho_{A}^{i}(q)y^{A} \\
        & -\mathcal{L}(q^{i},y^{A},\dot{y}^{A}(q^{i},y^{A},p_{A})).
    \end{split}
\end{equation}

Below we will see that the dynamics of the nonholonomic
optimal control problem is determined by the Hamiltonian system
given by the triple
$(T^{*}\mathcal{D},\omega_{\mathcal{D}},\mathcal{H})$, where
$\omega_{\mathcal{D}}$ is the standard symplectic $2$-form on
$T^{*}\mathcal{D}.$ The symplectic Hamiltonian dynamics is
determined by the dynamical equation 
\begin{equation}\label{dynamic}
i_{X_{\mathcal{H}}}\omega_{\mathcal{D}}=d\mathcal{H}.
\end{equation}  Therefore, if we consider the integral curves of $X_{\mathcal{H}},$ which are of the type $t\mapsto(q^{i}(t),y^{A}(t),p_{i}(t),p_{A}(t))$; the solutions of the nonholonomic Hamiltonian system is
specified by the Hamilton's equations on $T^{*}\mathcal{D}$:
\begin{eqnarray*}
\dot{q}^{i}&=&\tfrac{\partial\mathcal{H}}{\partial p_{i}},\qquad\quad \dot{y}^{A}=\tfrac{\partial\mathcal{H}}{\partial p_{A}},\\
\dot{p}_{i}&=&-\tfrac{\partial\mathcal{H}}{\partial q^{i}},\qquad\dot{p}_{A}=-\tfrac{\partial\mathcal{H}}{\partial y^{A}};
\end{eqnarray*} that is,
\begin{equation*}
\begin{array}{l}
\dot{q}^{i}=\rho_{A}^{i}y^{A},\\
\dot{p}_{i}=\dfrac{\partial\mathcal{L}}{\partial q^{i}}(q^{i},y^{A},\dot{y}^{A}(q^{i},y^{A},p_{A}))-p_j\dfrac{\partial\rho_{A}^{j}}{\partial q^{i}}y^{A}-p_A\dfrac{\partial \dot{y}^A}{\partial q^i},\\
\dot{p}_{A}=\dfrac{\partial\mathcal{L}}{\partial y^{A}}(q^{i},y^{A},\dot{y}^{A}(q^{i},y^{A},p_{A}))-p_{j}\rho_{A}^{j}-p_B\dfrac{\partial \dot{y}^B}{\partial y^A}.
\end{array}
\end{equation*}


\section{Geometric integrators for the optimal control of nonholonomic systems}\label{sec:retractions:D}
	
	The framework for the construction of geometric integrators is established by Proposition 5.1 in \cite{21MBLDMdD} which reads:
	
	\begin{proposition}\label{Prop:symplectic:integrator}
		If $R_{d}$ is a discretization map on $Q$ and $H:T^{*}Q \rightarrow \R$ is a Hamiltonian function, then the equation
        \begin{align*}
            &(R_{d}^{T^*})^{-1}(q_{0},p_{0},q_{1},p_{1}) =\\ & \sharp_{\omega_Q} \left( h dH \left[ \tau_{T^{*}Q} \circ (R_{d}^{T^*})^{-1}(q_{0},p_{0},q_{1},p_{1}) \right] \right)
        \end{align*}
        written for the cotangent lift of $R_{d}$ is a symplectic integrator.
	\end{proposition}

	The previous proposition might be adapted to our case since the Hamiltonian function is defined on $T^{*}\mathcal{D}$. As a result, we construct a symplectic integrator for the Hamiltonian version of the optimal control of nonholonomic systems. The only challenge might be to construct a discretization map on $\mathcal{D}$.
    
    It is clear that if one chooses local coordinates $(q^{i},y^{A})$ on $\mathcal{D}$ associated with a choice of local sections $\{e_{A}\}$ of the distribution $\mathcal{D}$, then one can introduce euclidean local retraction maps, such as the midpoint rule, on the domain of the coordinate chart. Given an open set $\hat{U} \subseteq Q$, we define a sufficiently small tubular neighborhood $U \subseteq \tau_{\mathcal{D}}^{-1}(\hat{U})$ of $\mathcal{D}$ containing the zero section. For any $\delta\in [0,1]$, we can define a local retraction map on $U$, denoted by $R_{U}:TU\to U\times U$, whose local expression is given by
    \begin{equation*}
            \begin{split}
                R_{U}(q^{i}, y^{A}, \dot{q}^{i}, \dot{y}^{A})=& (q^{i} -\delta \dot{q}^{i}, y^{A}-\delta  \dot{y}^{A} , \\
                & q^{i} +(1-\delta) \dot{q}^{i}, y^{A}+(1-\delta) \dot{y}^{A} ),
            \end{split}
        \end{equation*}
        and provided that the set $U$ is sufficiently small this map is well-defined.
    
    Alternatively, if we do not want to restrict ourselves to a coordinate chart in $\mathcal{D}$, we can consider the inclusion map $i_{\mathcal{D}}:\mathcal{D}\hookrightarrow TQ$, a Riemannian metric on $Q$ and the associated orthogonal projection $\mathcal{P}:TQ\rightarrow \mathcal{D}$. Then, given a discretization map of the type $R_{d}:TTQ\rightarrow TQ\times TQ$, we can define the map $R_{\mathcal{D}, d}:T\mathcal{D}\rightarrow \mathcal{D}\times \mathcal{D}$ as in Figure~\ref{fig:enter-label}.
    \begin{figure}[htb!]
        \centering
        \begin{tikzcd}
            TTQ \arrow[r, "R_{d}"]                                                              & TQ\times TQ \arrow[d, "\mathcal{P}"', bend right] \arrow[d, "\mathcal{P}", bend left, shift left] \\
            T\mathcal{D} \arrow[r, "{R_{\mathcal{D}, d}}"'] \arrow[u, "Ti_{\mathcal{D}}", hook] & \mathcal{D}\times \mathcal{D}                                                                    
        \end{tikzcd}
        \caption{Commutative diagram defining $R_{\mathcal{D}, d}=(\mathcal{P}\times \mathcal{P})\circ R_{d} \circ Ti_{\mathcal{D}}.$}
        \label{fig:enter-label}
    \end{figure}
    
    \begin{proposition}
        The map $R_{\mathcal{D}, d}$ is a discretization map on $\mathcal{D}$.
    \end{proposition}

    \begin{proof}
        On one hand, it is straightforward to check that $
R_{\mathcal{D}, d}(0_{v_q}) = (v_q, v_q)$ for all $v_{q}\in \mathcal{D}_{q}$.
Indeed,

$$\begin{aligned}
R_{\mathcal{D}, d}(0_{v_q}) &= (\mathcal{P} \times \mathcal{P}) \circ R_{d} \circ Ti_{\mathcal{D}} (0_{v_q})\\[1mm]
&= (\mathcal{P} \times \mathcal{P}) \circ R_{d} \bigl(0_{i_{\mathcal{D}}(v_q)}\bigr)\\[1mm]
&= (\mathcal{P} \times \mathcal{P})\bigl(i_{\mathcal{D}}(v_q),\, i_{\mathcal{D}}(v_q)\bigr)\\[1mm]
&= \Bigl(\mathcal{P}\bigl(i_{\mathcal{D}}(v_q)\bigr),\, \mathcal{P}\bigl(i_{\mathcal{D}}(v_q)\bigr)\Bigr)\\[1mm]
&= (v_q, v_q).
\end{aligned}$$

On the other hand, the second condition, that is, the equality
$T_{0_{v_q}}R_{\mathcal{D}, d}^{2} - T_{0_{v_q}}R_{\mathcal{D}, d}^{1} = \mathrm{id}_{T D},$
might be checked using the definition of tangent map in terms of the derivative of a curve. Let $X_{v_q}\in T_{v_{q}}\mathcal{D}$. We have that
\[
\begin{aligned}
\Big(  & \left.T_{0_{v_q}}R_{\mathcal{D}, d}^{2} - T_{0_{v_q}}R_{\mathcal{D}, d}^{1}\right)(X_{v_q}) \\
&= \frac{d}{ds}\Bigg|_{s=0} \Bigl[ R_{\mathcal{D}, d}^{2}(sX_{v_q}) - R_{\mathcal{D}, d}^{1}(sX_{v_q}) \Bigr] \\[1mm]
&= \frac{d}{ds}\Bigg|_{s=0} \Bigl[ \mathcal{P} \circ R_{d}^{2} \circ T i_{\mathcal{D}}(sX_{v_q}) - \mathcal{P} \circ R_{d}^{1} \circ T i_{\mathcal{D}}(sX_{v_q}) \Bigr] \\[1mm]
&= T \mathcal{P} \left[ \frac{d}{ds}\Bigg|_{s=0} \Bigl( R_{d}^{2} \circ T i_{\mathcal{D}}(sX_{v_q}) - R_{d}^{1} \circ T i_{\mathcal{D}}(sX_{v_q}) \Bigr) \right] \\[1mm]
&= T \mathcal{P} \Bigl[\bigl(T_{0_{v_q}} R_{d}^{2} - T_{0_{v_q}}R_{d}^{1}\bigr)\bigl(T i_{\mathcal{D}}(X_{v_q})\bigr) \Bigr] \\[1mm]
&= T \mathcal{P} \circ T i_{\mathcal{D}} (X_{v_q})=T \left( \mathcal{P} \circ  i_{\mathcal{D}} \right) (X_{v_q}) \\[1mm]
&= X_{v_q},
\end{aligned}
\]
where in the above we have used that \(\mathcal{P} \circ i_\mathcal{D} = \mathrm{id}_\mathcal{D}\). Hence, \(R_{\mathcal{D}, d}\) satisfies the defining properties of a discretization map on \(\mathcal{D}\).
    \end{proof}

    The previous construction prescribes a recipe to obtain global discretization maps on $\mathcal{D}$ extending beyond the domain of a coordinate chart.

\section{Integrator for controlled Chaplygin system}\label{sec:example}

We want to study the optimal control of the so-called
\textit{Chaplygin sleigh}. The sleigh is a rigid body moving on a horizontal plane supported at three points, two of which slide freely without friction, while the third is a knife edge which allows no motion orthogonal to its direction.

We assume that the sleigh cannot move sideways. The configuration space will be identified with  $Q=\R^2\times\mathbb{S}^1$ with coordinates $q=(x,y,\theta)$, where $\theta$ is the angular orientation of the sleigh, and $(x, y)$ are the position of the contact point of the sleigh on the plane, . Let $m$ be the mass of the sleigh and $J+ma^{2}$ is the inertia about the contact point, where $J$ is the moment of inertia about the center of mass $C$ and $a$ is the distance from the center of mass to the knife edge. 

The control inputs are denoted by $u^{1}$ and $u^{2}.$ The first one
corresponds to a force applied perpendicular to the center of mass
of the sleigh and the second one is the torque applied about the
vertical axis.

The constraint is given by the no slip condition and is
expressed by $\sin\theta\, \dot{x}-\cos\theta \dot{y} = 0.$  Therefore the constraint
distribution $\mathcal{D}$ is given by the span of the vector fields
\begin{equation*}
X_{1}(q)=\frac{1}{J}\frac{\partial}{\partial\theta}, \quad
X_{2}(q)=\frac{\cos\theta}{m}\frac{\partial}{\partial x}+\frac{\sin\theta}{m}\frac{\partial}{\partial y}.
\end{equation*}

\subsection{The dynamics}

The Lagrangian $L:T(\R^2\times\mathbb{S}^{1})\ra\R$ is exclusively given by the kinetic energy of the body, which is a sum
of the kinetic energy of the center of mass and the kinetic energy
due to the rotation of the body $$L(q,\dot{q})=\frac{m}{2}(\dot{x}_C^{2}+\dot{y}_C^{2})+\frac{J}{2}\dot{\theta}^{2},$$ where $x_C = x+a\cos\theta$, $y_C = y+a\sin\theta$. 



The basis $\{X_{1},X_{2}\}$ induces adapted coordinates $(x,y,\theta,z^1,z^{2})\in\mathcal{D}$ in the following way: given the  vector fields $X_1$ and $X_2$ generating the distribution  we obtain the relations for $q\in \R^2\times\mathbb{S}^1$
\begin{eqnarray*}
X_1(q)&=&\rho_1^1(q)\frac{\partial}{\partial x}+\rho_1^2(q)\frac{\partial}{\partial y}+\rho_1^3(q)\frac{\partial}{\partial\theta},\\
X_2(q)&=&\rho_2^1(q)\frac{\partial}{\partial x}+\rho_2^2(q)\frac{\partial}{\partial y}+\rho_2^3(q)\frac{\partial}{\partial\theta}.
\end{eqnarray*} Then, $$\rho_{1}^{1}= \rho_{1}^{2}=\rho_{2}^{3}=0,\quad\rho_{1}^{3}=\frac{1}{J},\quad
\rho_{2}^{1}=\frac{\cos\theta}{m},\quad\rho_{2}^{2}=\frac{\sin\theta}{m}.
$$ Each element $v\in {\mathcal D}_q$ is expressed as a linear combination of these vector fields:
\(
v=z^1 X_1(q)+ z^{2} X_2(q), \ q\in \R^2\times\mathbb{S}^1.
\)

Therefore, the vector subbundle $\tau_{\mathcal{D}}:{\mathcal D}\rightarrow \R^2\times\mathbb{S}^1$ is locally described by the coordinates $(x,y,\theta; z^1, z^{2})$; the first three, denoted by $q^{i}$, for the base and the last two, denoted by $y^{A}$, for the fibers. 
As a consequence, ${\mathcal D}$ is described by the conditions: 
$$
\dot{x}=\frac{\cos\theta}{m}z^{2},\quad
\dot{y}=\frac{\sin\theta}{m}z^{2},\quad
\dot{\theta}=\frac{1}{J}z^1,
$$
as  a vector subbundle of $TQ$.

The restricted Lagrangian function in the new adapted coordinates is
given by
$$\ell(x,y,\theta,z^1,z^{2})=\frac{1}{2m}(z^{2})^2+\frac{b}{2J}(z^1)^2 \hbox{ where } b=\frac{a^2m+J}{J}.$$

Therefore, the equations of motion are
$$
\dot{z}^1=0,\quad \dot{z}^{2}=0,\quad
\dot{x}=\frac{\cos\theta}{m}z^{2},\quad\dot{y}=\frac{\sin\theta}{m}z^{2},\quad\dot{\theta}=\frac{1}{J}z^1.
$$

\subsection{The optimal control problem}

Now, by adding controls into our picture, the controlled Euler-Lagrange equations are written as
$$\dot{z}^1=u^{2},\; 
\dot{z}^{2}=u^{1},\; 
\dot{x}=\frac{\cos\theta}{m}z^{2},\; \dot{y}=\frac{\sin\theta}{m}z^{2},\; \dot{\theta}=\frac{1}{J}z^1.
$$

The optimal control problem consists on finding an admissible curve
satisfying the previous equations given boundary conditions on
$\mathcal{D}$ and minimizing the functional
$\mathcal{J}(x,y,\theta,z^1,z^{2},u^{1},u^{2})=\frac{1}{2}\int_{0}^{T}\left((u^{1})^{2}+(u^{2})^{2}
\right)dt$, for the cost function $C:\mathcal{D}\times U\ra\R$
given by
\begin{equation}\label{costchaplygin}C(x,y,\theta,z^1,z^{2},u^{1},u^{2})=\frac{1}{2}((u^{1})^2+(u^{2})^2).\end{equation}

As before, the optimal control problem is equivalent to solving the
constrained  optimization problem determined by
$\mathcal{L}:\mathcal{D}^{(2)}\ra\R,$ where

$$
\mathcal{L}(x,y,\theta,z^1,z^{2},\dot{z}^1,\dot{z}^{2})=\frac{1}{2}\left((\dot{z}^1)^2+(\dot{z}^{2})^2\right).$$ Here, $\mathcal{D}^{(2)}$ is a submanifold of the vector bundle $T\mathcal{D}$ over $\mathcal{D}$ defined by $(x,y,\theta,z^1,z^{2},\dot{x},\dot{y},\dot{\theta},\dot{z}^1,\dot{z}^{2})\in T\mathcal{D}$ satisfying
$$\dot{x}-\frac{\cos\theta}{m}z^{2}=0,\dot{y}-\frac{\sin\theta}{m}z^{2}=0,\dot{\theta}-\frac{1}{J}z^1=0.$$

Denoting by $(x,y,\theta,z^1,z^{2},p_x,p_y,p_{\theta},p_1,p_2)$ local
coordinates on $T^{*}\mathcal{D}$ the dynamics of the optimal control
problem for this nonholonomic system is determined by the
Hamiltonian function $\mathcal{H}:T^{*}\mathcal{D}\ra\R,$
{$$ 
\mathcal{H}=\frac{1}{2}(p_{1}^{2}+p_2^{2})+\frac{p_{\theta}}{J}z^{1}+p_{x}\frac{\cos\theta}{m}z^{2}+p_{y}\frac{\sin\theta}{m}z^{2}.$$}

The Hamiltonian equations of motion are

\begin{eqnarray*}
\dot{z}^{1}&=&p_{1},\quad\dot{z}^{2}=p_{2},\quad\dot{p}_{x}=0,\quad \dot{p}_{y}=0,\\
\dot{p}_{\theta}&=&p_{x}\frac{\sin\theta}{m}z^{2}-p_{y}\frac{\cos\theta}{m}z^{2},\\
\dot{p}_{1}&=&-\frac{p_{\theta}}{J},\quad\dot{p}_{2}=-p_{x}\frac{\cos\theta}{m}-p_{y}\frac{\sin\theta}{m}.
\end{eqnarray*}

\subsection{Construction of a retraction map}\label{sec6}

    Consider a retraction map on $T(\R^{2}\times \mathbb{S}^{1})$ denoted by $R_{d}^{T}$ of the form
    $$R_{d}^{T}(q,v,\dot{q}, \dot{v})=(q-\delta \dot{q}, v-\delta \dot{v}, q+(1-\delta)\dot{q}, v+(1-\delta) \dot{v}).$$

    Let us consider first on the domain of the coordinate chart $(x,y,\theta,z^{1}, z^{2})$. Noting that $\mathcal{D}$ is a distribution on $\R^{2}\times \mathbb{S}^{1}$, the inclusion map has coordinate expression
    $$i_{\mathcal{D}}(x,y,\theta, z^{1}, z^{2})=\left(x,y,\theta, \frac{\cos \theta}{m} z^{2}, \frac{\sin \theta}{m}z^{2}, \frac{z^{1}}{J}\right)$$
    and the orthogonal projection which in natural coordinates $(x,y,\theta,\dot{x},\dot{y}, \dot{\theta})$ reads
    $$\mathcal{P}(x,y,\theta,\dot{x},\dot{y}, \dot{\theta})= J \dot{\theta} X_{1} + m (\dot{x} \cos \theta + \dot{y} \sin \theta) X_{2}.$$
    Hence, we can compute the discretization map $R_{\mathcal{D}, d} = (\mathcal{P}\times \mathcal{P}) \circ R_{d}^{T} \circ T i_{\mathcal{D}}$:
    \begin{equation*}
        \begin{split}
            R_{\mathcal{D},d}(q^{i}, y^{A}, \dot{q}^{i}, \dot{y}^{A}) = &( q-\delta\dot{q},  y^{A}-\delta\dot{y}^{A}, \\
             & q + (1-\delta)\dot{q}, y^{A}+(1 - \delta)\dot{y}^{A}).
        \end{split}
    \end{equation*}
    
    Note that, in general, given the discretization $R_{d}^T$ above on natural fiber bundle coordinates on $TQ$, there is no reason why the discretization $R_{\mathcal{D},d}$ should preserve the same form in the adapted chart on $\mathcal{D}$. In fact, in the general situation there should appear derivatives of the function $\rho_{A}^{i}$ with respect to configuration variables $q^{j}$. However, in this example, these terms cancel out.

    The symplectic integrator for $\mathcal{H}$ uses Proposition~\ref{Prop:symplectic:integrator} and $\left(R_{\mathcal{D},d}\right)^{T^*}$ to generate the following symplectic numerical scheme on $T^*\mathcal{D}$: 
    \begin{equation*}
        \begin{split}
            & \frac{q_{1}^{i}-q_{0}^{i}}{h} =  \tilde{\rho}^{i}_{A}\frac{y_{1}^{A}+y_{0}^{A}}{2},\quad
            \frac{y_{1}^{A}-y_{0}^{A}}{h} =  \frac{p_{A,1}+p_{A,0}}{2},\\
            & \frac{p_{i,1}-p_{i,0}}{h} = -\frac{p_{j,0}+p_{j,1}}{2}\frac{\partial \tilde{\rho}_{A}^{j}}{\partial q^{i}}\frac{y_{1}^{A}+y_{0}^{A}}{2},\\
            & \frac{p_{A,1}-p_{A,0}}{h} = -\frac{p_{j,0}+p_{j,1}}{2}\tilde{\rho}_{A}^{j}.
        \end{split}
    \end{equation*}
	where $\tilde{\rho}_{A}^{i} = \rho_{A}^{i}(\tfrac{q_{0}+q_{1}}{2})$  and $\frac{\partial \tilde{\rho}_{A}^{j}}{\partial q^{i}}=\frac{\partial \rho_{A}^{j}}{\partial q^{i}}(\tfrac{q_{0}+q_{1}}{2})$. The resulting numerical scheme is by construction on $\mathcal{D}$. However, in the following experiments we will test to what extent the sequence of configuration variables $q_{k}=(x_{k}, y_{k}, \theta_{k})$ is preserving a discrete version of the constraints to check if the discrete velocities are compatible with them. For that matter, we will test the preservation of the function
    \begin{equation}\label{eq:phi_d}
        \begin{split}
            \phi_{d}(q_{k},q_{k+1})=(x_{k+1} & -x_{k})\sin\left(\frac{\theta_{k+1}+\theta_{k}}{2}\right) \\
            & - (y_{k+1}-y_{k})\cos\left(\frac{\theta_{k+1}+\theta_{k}}{2}\right).
        \end{split}
    \end{equation}
    
    \subsection{Numerical results}

The numerical simulation has the initial conditions $(x_0,y_0, \theta_0, z^1_0, z^{2}_0, p_{x, 0}, p_{y, 0}, p_{\theta, 0}, p_{1, 0}, p_{2, 0})$ set to $(1,1,\pi, 0.05, 0.05, 0, 1, 0, 0, 0)$ and ran for a time span of $20$ units of time with a fixed time step of $h=0.005$. The constants $m$, $J$ and $b$ were set to $1$.

We have compared five different algorithms: two integrators based on Proposition \ref{Prop:symplectic:integrator} using the retraction map $R_{\mathcal{D},d}$ of the form above; a second-order explicit Runge-Kutta method; a fourth-order explicit Runge-Kutta method; and, finally, a fourth-order symplectic Runge-Kutta method.

The first integrator is based on the midpoint discretization map, i.e., the map obtained by setting $\delta=1/2$. The second one is the composition of the initial point discretization map, obtained by setting $\delta=0$ with the final point discretization map, obtained by setting $\delta=1$. Indeed, as pointed out in \cite{DM}, the composition of different discretization maps holds higher-order symplectic numerical integrators. This particular choice coincides with the well-known Störmer-Verlet method for Hamiltonian systems.

One of the advantages of symplectic integrators is the excellent energy preservation and, as a by-product, a good qualitative behavior during the integration over long periods of time. However, as mentioned in the report \cite{chyba}, in optimal control problems the time span of integration is usually relatively short and it is natural to question the importance of symplectic integrators in this setting.

In Figure~\ref{rk:constraint:plot}, we compare the preservation of the discrete constraint $\phi_{d}$ in Equation~\eqref{eq:phi_d} using the two Runge-Kutta methods and the retraction based method with $\delta=1/2$. We first observe that the constraint preservation is even better than the higher-order Runge-Kutta at the chosen time-step $h=0.005$.

\begin{figure}[htb!]
    \centering
    \includegraphics[scale=0.35]{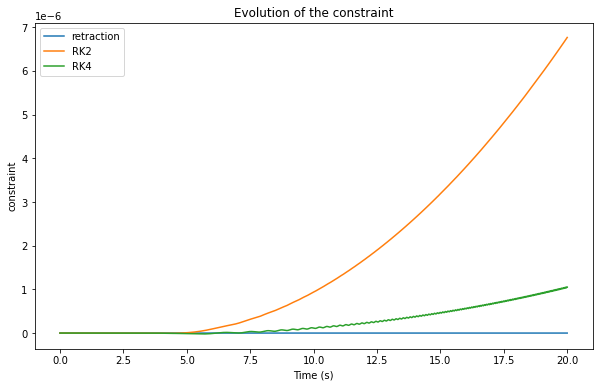}
    \caption{Plot of the discrete constraint function $\phi_{d}$ as the number of iterations increases. The plot depicts standard second and fourth-order Runge-Kutta methods, against the retraction integrator with $\delta=1/2$.}
    \label{rk:constraint:plot}
\end{figure}

If we zoom in and compare the behavior of the two symplectic integrators, i.e., the retraction-based method with $\delta=1/2$ and the Störmer-Verlet method, the constraint preservation in the latter is even more spectacular (see Figure \ref{rk:constraint:plot:symp})

\begin{figure}[htb!]
    \centering
    \includegraphics[scale=0.35]{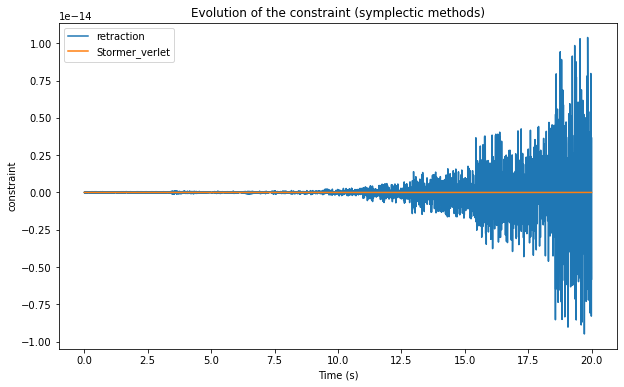}
    \caption{Plot of the discrete constraint function $\phi_{d}$ as the number of iterations increases. The plot depicts the two (symplectic) integrators based on retractions.}
    \label{rk:constraint:plot:symp}
\end{figure}

Despite the excellent constraint preservation, the main factor for the preservation of the Hamiltonian function seems to be the order of the method. In Figure \ref{rk:energy:plot:order2}, we can see that the energy is much better preserved by the higher-order Runge-Kutta method than by the (symplectic) retraction integrator.

\begin{figure}[htb!]
    \centering
    \includegraphics[scale=0.35]{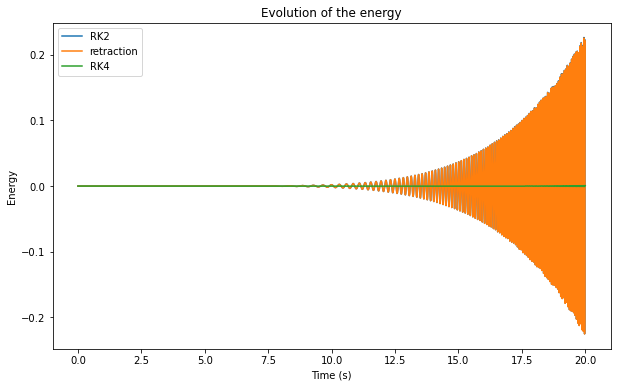}
    \caption{Plot of the Hamiltonian function (energy) as the number of iterations increases. The plot depicts a symplectic fourth order method against a standard fourth-order Runge-Kutta method.}
    \label{rk:energy:plot:order2}
\end{figure}

If the symplectic method is of fourth-order, such as the one in Figure \ref{rk:energy:plot}, then the energy performance of the symplectic method is slightly better than the non-symplectic.

\begin{figure}[htb!]
    \centering
    \includegraphics[scale=0.35]{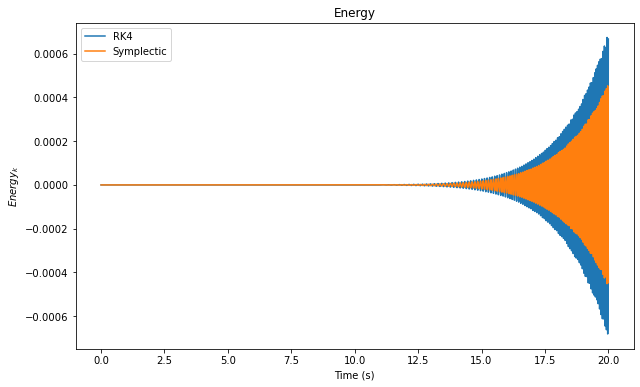}
    \caption{Plot of the Hamiltonian function (energy) as the number of iterations increases. The plot depicts a symplectic fourth order method against a standard fourth-order Runge-Kutta method.}
    \label{rk:energy:plot}
\end{figure}

However, the precise reason why the symplectic methods are not preserving the Hamiltonian function much better than the non-symplectic methods deserves further study. Our best hypothesis is that the maximum time-step for which backward error analysis guarantees that the modified Hamiltonian is preserved must be unreasonably small (see \cite{hairer}), given that we have experimented with time steps of the order of $h\approx 10^{-5}$ and the behavior persists. For larger time steps, the most important attribute controlling the error in the Hamiltonian function seems to be the order of the numerical scheme.



\section{Conclusions \& Further applications}
In this paper, we have designed geometric integrators for optimal control problems of fully-actuated nonholonomic systems that preserve symplecticity of PMP equations and exactly preserve the nonholonomic constraints. Our method also opens the door to study control problems for underactuated nonholonomic problems and to explore methods preserving the symmetry of the initial system.

\end{document}